\documentclass[11pt]{amsart}

\usepackage{graphicx}
\usepackage{epsfig}
\usepackage{color}
\definecolor{r}{rgb}{0.9,0.3,0.1}
\definecolor{b}{rgb}{0.1,0.3,0.9}

\usepackage{color}
\usepackage{amsmath}
\usepackage{amssymb}
\usepackage{amsthm}
\usepackage{enumerate}
\usepackage{amsbsy}
\usepackage{amsfonts}
\newtheorem{theo}{Theorem}[section]
\newtheorem{defin}[theo]{Definition}
\newtheorem{prop}[theo]{Proposition}

\newtheorem{lemm}[theo]{Lemma}
\newtheorem{rem}[theo]{Remark}

\newcommand{\al}{\alpha}

\newcommand{\Ga}{\Gamma}

\newcommand{\Om}{\Omega}

\newcommand{\De}{\Delta}
\newcommand{\de}{\delta}

\newcommand{\pa}{\partial}

\newcommand{\R}{{\bf R}^n}

\newcommand{\ri}{\rightarrow}

\begin{document}
\baselineskip=18pt

\title[Fractional Laplace]{A Dirichlet problem of the fractional Laplace equation 
in the bounded Lipschitz domain}
\author{TongKeun Chang}
\address{Department of Mathematics, yonsei University, Seoul, South Korea}
\email{chang7357@yonsei.ac.kr}

%\thanks{This paper is supported by }

\begin{abstract}
In this paper, we study a Dirichlet problem of a fractional Laplace
equation in a bounded Lipschitz domain in $ \R, \,\, n \geq 2$. Our
main result is that for the given data $F \in \dot H^s(\Om^c), \,\,
0 < s<1$, we find the function which satisfies that $\De^s u =0$ in
$\Om$,
 $u|_{\Om^c} =F$ and $\| u\|_{\dot{H}^s(\R)} \leq c \| F\|_{\dot H^s(\Om^c)}$. Furthermore, we represent the
 solution with an integral operator.\\

\end{abstract}

\maketitle

\section{Introduction}
\setcounter{equation}{0}
In this paper,  we study the Dirichlet problem of a fractional Laplace equation
for the bounded Lipschitz domain $\Om$ in $\R, \,\, n \geq 2$. For $0< s < 1$,
 we define the fractional Laplacian of a function $u: \R \ri {\bf R}$ as
\begin{align}\label{definition2}
\De^s u(x) : = c(n,s) \int_{\R} \frac{ u(x+y) - 2 u(x) + u(x-y)}{|y|^{n + 2s}} dy,
\end{align}
where $c(n,s)$ is some normalization constant.
The fractional Laplacian of $u$ also can be defined as a pseudo-differential operator
\begin{align}\label{definition3}
\widehat{(-\De)^s  u}(\xi) = (2\pi  | \xi|)^{2s} \hat{u} (\xi),
\end{align}
where $\hat{ u}(\xi) :=\int_{\R} u(x) e^{-2\pi i \xi \cdot x } dx, \,\, \xi \in \R$
is the Fourier transform of $u$ in $\R$.

Compared with the classical Laplacian of $u$, $\De u = \sum_{1 \leq
i \leq n} \frac{\pa^2 u}{\pa x_i \pa x_i} $ is a local property but
the fractional Laplacian of $u$ ( \eqref{definition2} or
\eqref{definition3}) is a non-local property. That is, to define the
classical Laplacian of $u$ at $x\in \R$, we only need the
information of $u$ in the neighborhood of $x$, but to define the
fractional Laplacian of $u$ at $x$, we also need  the information in
$\R$. So, to establish the  Dirichlet problem of the fractional
Laplace equation in $\Om$, we need the condition which is defined in
$\Om^c$.

We introduce the Dirichlet problem of a fractional Laplace equation
in the bounded Lipschitz domain; given the function $F$ defined in
$\Om^c$, we find the function $u$ satisfying the following equation
\begin{align}\label{main equation}
\left\{\begin{array}{l} \vspace{2mm}
\De^s  u(x)  =0  \qquad  x \in \Om,\\
  u(x) = F(x) \qquad x \in \Om^c,
 % \| u\|_{H^s(\R)} \leq c \| F\|_{H^s(\Om^c)}.
\end{array}
\right.
\end{align}

The  probability theory is a good  tool to represent the solution of \eqref{main equation}.
Let $X_t$ be a $2s$-stable process in
$\R$ and $\tau_\Om = \inf\{ t> 0\, | \, X_t \notin \Om\}$, the first exit time
 of $X_t$ in $\Om$. Note that
$X_t$ is right continuous and has left limits $a.s.$ Furthermore
$P\{\tau_\Om \in \pa \Om\} =0$  (see \cite{B} and \cite{CS}) and so
$X_{\tau_\Om} \in \Om^c, \, a.s$. Let  $F \in C^\infty(\Om^c)$ with
$\int_{\Om^c} |F(x)|(1 + |x|)^{-n -2s} dx < \infty$ and we define
the function
\begin{align}\label{probability solution}
 u(x) = \left\{ \begin{array}{ll}
E_x F(X_{\tau_\Om}) & \quad x \in \Om, \\
F(x) & \quad x \in \Om^c,
\end{array}
\right.
\end{align}
where $E_x$ denotes an expectation with respect to $P^x$  of the
process starting from $x \in \Om$. Then,  $u$ defined in
\eqref{probability solution}
 is the solution of \eqref{main equation}.
Moreover, there is a Poisson kernel $K(x,y)$ defined in $\Om \times
\Om^c$ such that
\begin{align}\label{potentail0313}
u(x) = \int_{\Om^c} K(x,y) F(y)  dy.
\end{align}
(see \cite{B} and \cite{CS}).

In this paper, we study the regularity problem of the equation
\eqref{main equation}. Our main result is the following theorem.
\begin{theo}\label{theo0229}
Let $\Om$ be a bounded Lipschitz domain in $\R$ and $0 < s <1$. Then, given $F \in \dot{H}^s(\Om^c)$,
there is the unique weak solution $u \in \dot{H}^s(\R)$ of the equation \eqref{main equation} such that
\begin{align}\label{homo-estimate}
\| u\|_{\dot H^s(\R)} \leq c \| F\|_{\dot H^s (\Om^c)}
\end{align}
for some positive constant independent of $F$.
\end{theo}
The function spaces $\dot{H}^s(\R)$ are homogeneous Sobolev spaces
(see \eqref{h-sobolev}) and $\dot{H}^s(\Om^c)$ are restrictions of
$\dot{H}^s(\R)$ over $\Om^c$ (see \eqref{rest-sobolev}).

To show the theorem  \ref{theo0229}, we  use the Riesz potential. In
 section \ref{integral operator}, we will define the integral
operator $S_s : \dot{H}^{-s}_{ 0} (\Om^c) \ri \dot{H}^s (\Om^c)$ by
\begin{align*}
S_s \phi = (I_{2s}\tilde f)|_{\Om^c}, \quad f \in  \dot{H}^{-s}_{ 0} (\Om^c),
\end{align*}
where $\dot{H}^{-s}_{ 0} (\Om^c)$ is a dual space of $\dot{H}^s
(\Om^c)$, $I_{2s}$ is the  Riesz transform and $\tilde \phi$ is a
zero extension of $\phi \in \dot{H}^{-s}_{ 0} (\Om^c)$ (see
\eqref{zero extension}). Note  that $u = I_{2s}\tilde \phi$ is a
weak solution of  the equation \eqref{main equation} (see section
3). So, the existence of a solution of the equation \eqref{main
equation} is related with the bijectivity of the operator $S_s :
\dot{H}^{-s}_{ 0} (\Om^c) \ri  \dot{H}^s (\Om^c)$. In section 4, we
will show that $S_s : \dot{H}^{-s}_{ 0} (\Om^c) \ri  \dot{H}^s
(\Om^c)$ is bijective.

Related with the regularity problem of the fractional Laplace
equation, T. Chang\cite{Ch} showed the  Dirichlet problem  of the
fractional Laplace equation whose domain is $\R \setminus \pa \Om$
and boundary is $\pa \Om$  for a bounded Lipschitz domain $\Om$. He
showed the existence of the solution of a fractional Laplace
equation whose boundary data is  in $H^{s -\frac12}(\pa \Om)$ and
then the solution is in $\dot{H}^s(\R)$ such that $|u(x)|   =
O(|x|^{-n +2s} )$ at infinite. To show this, he showed the
bijectivity of the boundary integral operator induced from the Riesz
potential.

We introduce another equivalent definition of the weak solution of a
fractional Laplace equation in $\Om$. We say that a function $u : \R
\ri {\bf R}$ is a $2s$-harmonic in $\Om$ if
\begin{align}\label{al-harmonic}
u(x) = E_x u(X_{\tau_V})
\end{align}
 for every bounded open set $V$ whose closure $\bar V$ is contained in $\Om$.
We say that $u$ is  a regular
 $2s$-harmonic in $\Om$ if
 \begin{align}\label{regular al-harmonic}
 u(x) = E_x u(X_{\tau_\Om}).
 \end{align}

By the strong Markov property of $X_t$,  if $u$ is a $2s$-harmonic
function  in $\Om$ and $V $ is a
 open subset of $\Om$ such that $\bar V \subset \Om$, then $u$ is a $2s$-regular harmonic in $V$. Moreover, \eqref{regular al-harmonic}
 implies \eqref{al-harmonic}, so that
regular $2s$-harmonic functions are $2s$-harmonic functions. The converse is not true (see section 3 in \cite{Bo2}).
It is well-known that $u$ is $2s$-harmonic in $\Om$ if and only if $u$ is continuous and
$\De^s u(x) =0$ for $x \in \Om$ (see Theorem 3.9  in \cite{Bo1}).

The second main result of this paper is the following theorem.
\begin{theo}\label{theo0229-2}
If  $\Om$ is a bounded $C^{1,1}$ domain and  $0 < s <
\frac{n}{2(n-1)}$, then the  weak solution $u$ of the theorem
\ref{theo0229} is regular $2s$-harmonic.
\end{theo}

This paper is organized as follows. In section 2, we introduce
several function spaces and study their properties. In section 3, we
introduce integral operators in function spaces defined in section
2. In section 4 and section 5, we prove the main results.

\section{Function spaces}\label{function space}
\setcounter{equation}{0}

In this paper, we denote $\Om$ is a bounded Lipschitz domain. we
also denote the letters $x,\, y, \, \xi$ as  points in $\R$. The
letter $c$ denotes a positive constant depending only on $n, \,\, s
$ and $\Om$.

To statement the main results, we introduce  several results
of harmonic analysis (see ,eg., chapter 9
in  \cite{Fo}).
Let ${\mathcal S}(\R)$ be Schwartz space and ${\mathcal S}'(\R)$ be its dual space
(space of distribution).
% Let $f, \, g \in {\mathcal S}(\R)$ and $\psi \in {\mathcal S}'(\R)$.
\begin{itemize}
\item[{\bf 1})](Fourier transform of distribution)
We define Fourier transform $\hat f \in {\mathcal S}'(\R)$ of $f$ by
\begin{align*}
< \hat f, \phi>: = <f, \hat \phi> \qquad \phi \in {\mathcal S}(\R),
\end{align*}
where $< \cdot, \cdot>$ is duality pairing between  ${\mathcal S}(\R)$ and ${\mathcal S}'(\R)$.
\item[{\bf 2})] (Convolution of distribution)
We define the convolution of $f \in {\mathcal S}^{'}(\R)$ and
$\phi \in {\mathcal S}(\R)$ by
\begin{align*}
f * \phi(x): = <f, \phi(x- \cdot)> \qquad x \in \R.
\end{align*}
Note that $f * \phi $ is slowly decay $C^\infty(\R)$ function and so
$f *\phi \in {\mathcal S}' (\R) $ is well-defined and
\begin{align*}
<f * \phi, \psi>  = \int_{\R} f * \phi(x) \psi(x) dx = <f, \phi * \psi>\quad
\psi \in {\mathcal S}(\R).
\end{align*}
\item[{\bf 3})]
For $f \in{\mathcal S}^{'}(\R)$ and $\psi, \, \phi \in {\mathcal S}(\R)$,
we have
\begin{align*}
f *(\psi * \phi)  =(f *\psi) *\phi =  (f * \phi) *\psi.
\end{align*}

\end{itemize}

Now, we are ready to define the function spaces.
Let $\eta$ be in Schwartz space ${\mathcal
S} ({\bf R}^{n})$ such that
\begin{eqnarray*}
 \left\{\begin{array}{rl}
&\hat \eta(\xi) > 0  \quad \mbox{ on }  2^{-1} < |\xi|   < 2 ,\\
& \hat \eta(\xi)=0 \quad \mbox{ elsewhere}, \\
&\sum_{ -\infty < i < \infty } \hat \eta(2^{-i}\xi ) =1 \qquad \xi  \neq 0.
\end{array}\right.
\end{eqnarray*}
We define functions  $\eta_i \in {\mathcal S}({\bf R}^{n}) $
whose Fourier transforms are
 written by
\begin{eqnarray}\label{psi1}
\begin{array}{ll}
\widehat{\eta_i}(\xi ) := \hat \eta(2^{-i} \xi ) \quad (i = 0, \pm 1, \pm 2 , \cdots).
\end{array}
\end{eqnarray}
%Note that $\eta_i (x) = 2^{in} \eta(2^i x )$.
For $\al \in {\bf R}$, we define the homogeneous Sobolev space
${\dot H}^{\al } ({\bf R}^{n})$ by
\begin{eqnarray}\label{h-sobolev}
{\dot H}^{\al }  ({\bf R}^{n}): = \{ f \in {\mathcal
S}^{'}({\bf R}^{n}) \, | \, \|f\|_{{\dot H}^{\al }(\R)} <
\infty \, \}.
\end{eqnarray}
Here,
\begin{align} \label{norm}
 \|f\|^2_{{\dot H}^{\al } (\R)} :&  =
\sum_{ -\infty < i  < \infty}  \int_{\R}    |\xi|^{2\al} |\widehat{\eta_i *
f}(\xi)|^2 d\xi,
\end{align} where $*$ is a convolution in ${\bf
R}^{n}$.
Note that $\|f\|_{{\dot H}^{\al } (\R)} =0$ if and only if $supp \, \hat f = \{0\}$, i.e. if and only if
$f$ is a polynomial (see chapter 6.3 in \cite{BL}).
By {\bf 2}), $\eta_i * f $ is slow decay $C^\infty$ function and so \eqref{norm}
is well-defined for $\al < 0$ also.

For $\al \geq 0$,
\begin{align*}
\sum_{ -\infty < i  < \infty} \int_{\R}|\xi|^{ 2\al } |\widehat{\eta_i * f}(\xi)|^2 d\xi
&= \sum_{ -\infty < i  < \infty} \int_{\R}
           |\xi|^{ 2\al } |\hat \eta_i(\xi)|^2 |\hat f(\xi)|^2 d\xi\\
&\approx  \int_{\R}
           |\xi|^{ 2\al }   |\hat f(\xi)|^2 d\xi,
\end{align*}
where $"A \approx B"$ means that there are positive constants $c_1$
and $c_2$ such that $c_1 A \leq B \leq c_2 A$. Hence, for $\al \geq
0$, we get
\begin{equation}\label{homogeneous norm}
\dot{H}^\al(\R) = \{ f \in{\mathcal S}'(\R) \,\, | \,\, \int_{\R} |\xi|^{2\al} |\hat{ f}(\xi)|^2 d\xi < \infty \}.
\end{equation}
and
\begin{align*}
\| f\|_{\dot{H}^\al(\R)} \approx   \Big(\int_{{\bf R}^n}|\xi|^{2\al}|\hat{ f}(\xi)|^2 d\xi \Big)^\frac12.
\end{align*}

By simple calculation,   for $f \in \dot{H}^\al(\R), \,\, 0 < \al <2$,  we obtain
\begin{align} \label{homogeneous norm2}
\int_{{\bf R}^n} \int_{{\bf R}^n}
\frac{|f(x+y )- 2f(x) +  f( x-y)|^2}{|y|^{n + 2\al}} dydx =C\int_{{\bf R}^n}|\xi|^{2\al}|\hat{f}(\xi)|^2 d\xi.
\end{align}

Note that for $\al > 0$,  $\dot{H}^{-\al}(\R) $
 is  dual space of $\dot{H}^{\al}(\R)$.
That is, $\dot{H}^{-\al}(\R) = (  \dot{H}^{\al}(\R))^*$.
Note that for  $\phi \in \dot{H}^{-\al}(\R)$ and $ f \in  \dot{H}^{\al}(\R)$,
\begin{align}\label{dual}
< \phi, f> : = \sum_{-\infty < i < \infty} \int_{\R} \widehat{\eta_i *\phi}(\xi) \hat{f}(\xi) d\xi.
\end{align}

Suppose that $\psi \in{\mathcal S} (\R)$ such that
$supp \,\, \hat \psi \cap supp \,\, \hat\eta_i =\emptyset$.
Then, by {\bf 1}) and {\bf 2}), we have
\begin{align*}
< \widehat{\eta_i * f}, \psi> = < \eta_i * f, \hat \psi> = < f, \eta_i * \hat \psi>
%\int_{\R} \widehat{\eta_i * f}(\xi) \psi(\xi) d\xi &= \int_{\R} \widehat{\eta_i * f}(\xi) \hat \Psi(\xi) d\xi\\
% & =  \int_{\R}  \eta_i * f(x)   \hat \phi (x) dx\\
% & = < f, \eta_i * \hat \psi>.
\end{align*}
But,  by Plancherel's Theorem, we get
\begin{align*}
\eta_i * \hat \psi(x) & = \int_{\R} \eta_i(y)\hat \psi(x-y) dx\\
 & = \int_{\R} e^{-2\pi x \cdot \xi} \hat \eta_i(\xi) \psi(\xi) d\xi\\
 & =0.
\end{align*}
This implies that $supp \,\, \widehat{f * \eta_i} \subset supp \,\, \hat\eta_i$.
For $f \in \dot H^\al(\R), \,\, \al \in \R$, we define the fractional Laplacian of $f$ by
\begin{align*}
\widehat{\De^s f} = (2\pi |\xi|)^{2s} \hat f.
\end{align*}
If $\al > 2s$, then above definition is same with \eqref{definition3}.
Since for
$f \in \dot H^\al(\R), \,\, \al \in {\bf R}$,
$f = \sum_{-\infty< i <\infty} f * \eta_i$,
by {\bf 3}), we have
\begin{align*}
\eta_i * (\De^s f) & = {\mathcal F}^{-1}(\widehat{\eta_i * (\De^s f)})\\
 &=  {\mathcal F}^{-1} \Big(\sum_{-\infty< k < \infty}  (2\pi |\xi|)^{2s} \hat \eta_i
             \widehat{  \eta_k*f} \Big),
\end{align*}
where  ${\mathcal F}^{-1}$ be a inverse Fourier transform.
Since $supp\,\, \hat \eta_i \cap supp \,\, \hat \eta_k =\emptyset$ for $|i-k| >2$ and
$\eta_{i-1}(\xi) + \eta_i(\xi) + \eta_{i+1}(\xi) =1$ for
$\xi \in supp \,\, \eta_i$, we have
\begin{align*}
\eta_i * (\De^s f)  & =  {\mathcal F}^{-1} \Big((2\pi|\xi|)^{2s}  \hat \eta_i
  \Big(\widehat{  \eta_{i-1}*f}   +\widehat{  \eta_i*f}  +\widehat{  \eta_{i+1}*f}\Big) \Big)\\
 & = \De^s\Big( \Big(\eta_{i-1} + \eta_i + \eta_{i+1} \Big)  *  ( \eta_i*f )  \Big) \\
 & = \De^s  ( \eta_i*f ).
\end{align*}
Hence, we get
\begin{align*}
\|\De^s f\|^2_{\dot H^{\al -2s}(\R)} & =\sum_{-\infty<i<\infty}
\int_{\R}   |\xi|^{2(\al -2s)}|\widehat{\eta_i * \De^s f}(\xi)|^2 d\xi\\
 & =\sum_{-\infty<i<\infty}
\int_{\R}   |\xi|^{2(\al -2s)}|\widehat{ \De^s(\eta_i *  f)}(\xi)|^2 d\xi\\
 &=\sum_{-\infty<i<\infty}
\int_{\R}    |\xi|^{2\al }  |\widehat{ (\eta_i* f)}(\xi)|^2 d\xi\\
& = \| f\|^2_{\dot H^\al(\R)}
\end{align*}
Hence, $\De^s: \dot H^{\al}(\R) \ri \dot H^{\al -2s}(\R)$ is isomorphism.

Now, we consider   a   bounded Lipschitz domain $\Om$ in $\R, \,\, n \geq 2$.

\begin{defin}\label{1013-2}
Let  $ 0 <s <  1$. We say that $v \in \dot{H}^s  (\R)$ is a {\it weak solution} of
a fractional Laplace equation $\De^s$ in $\Om $
if  $ \De^sv =0$   in $\Om$. That is, $v$ satisfies
\begin{align}\label{1013}
< \De^s v,    \psi> = 0
\end{align}
for all $\psi \in \dot{H}^s(\R)$ whose support  is contained in  $\Om $.
\end{defin}

\begin{rem}\label{rem-0221}
\begin{itemize}
\item[(1)]
By the definition of the fractional Laplacian \eqref{definition3} and \eqref{dual},
for $f, \psi \in \dot{H}^s(\R)$
we have
\begin{align}\label{rem-0402}
\notag< \De^s f,    \psi> & = \sum_{-\infty< i < \infty}
             \int_{\R} \widehat{ \eta_i*\De^s f}(\xi) \hat \psi(\xi) d\xi  \\
  & =    \sum_{-\infty< i < \infty}        \int_{\R} \widehat{\De^s ( \eta_i *f)}(\xi) \hat \psi(\xi) d\xi \\
\notag& = \int_{\R} (2\pi  |\xi|)^{2s} \hat{f}(\xi) \bar {\hat{\psi}} (\xi) d \xi.
\end{align}

\item[(2)]
In fact, $v$ is  continuous function in $\Om$ and satisfies
\begin{align*}
\De^s v(x) =0 \qquad \mbox{for } \qquad x \in \Om.
\end{align*}
(see Theorem 3.9  in \cite{Bo1}).
\end{itemize}
\end{rem}

 For $\al > 0$, we define the function spaces
\begin{align}\label{rest-sobolev}
\begin{array}{ll} \vspace{2mm}
\dot{H}^\al(\Om^c): &= \{ f|_{\Om^c}  \, | \,  f \in \dot{H}^\al(\R) \},\\
\dot{H}^\al_{ 0}(\Om^c): &= \{ f \in L^2_{loc} (\Om^c) \, | \, \tilde f \in \dot{H}^\al(\R) \},
\end{array}
\end{align}
where $\tilde f $ is zero extension of $f$ to $\R$. That is,
\begin{align*}
\tilde f(x) = \left\{\begin{array}{ll} \vspace{2mm}
f(x)& x \in \Om,\\
0 & \mbox{otherwise}.
\end{array}
\right.
\end{align*}
 The norms are
\begin{align*}
\| f\|_{\dot{H}^\al(\Om^c)  } : & = \inf_{F \in \dot{H}^\al(\R), \, F|_{\Om^c} =f } \| F\|_{ \dot{H}^\al(\R)  },\\
\| f \|_{ \dot{H}^\al_{ 0}(\Om^c)} : & = \| \tilde f\|_{\dot{H}^\al(\R) }.
\end{align*}
We also define $\dot{H}^{-\al}(\Om^c), \,\, \dot{H}^{-\al}_{0}(\Om^c)$ are dual spaces of
$ \dot{H}^{\al}_{ 0}(\Om^c)$ and
$\dot{H}^{\al}(\Om^c)$, respectively.
For $f \in \dot{H}^{-\al}_{ 0}(\Om^c)$, we define
$\tilde f \in  \dot{H}^{-\al}(\R)$ by
\begin{align}\label{zero extension}
<\tilde f, \phi> : = <f, \phi|_{\Om^c}> \quad \phi \in \dot{H}^{\al}(\R).
\end{align}
Note that
\begin{align}\label{equiv}
\| \tilde \phi\|_{\dot{H}^{-\al}(\R)} \approx \| \phi\|_{\dot{H}^{-\al}_{0} (\Om^c)}.
\end{align}

\section{Integral operators}\label{integral operator}
\setcounter{equation}{0}
For $ 0 < s < n $, we define Riesz transform in $\dot{H}^\al(\R), \,\, \al \in {\bf R}$ by
\begin{align*}
I_s: \dot{H}^{\al} (\R) \ri  \dot{H}^{\al +s} (\R), \qquad \widehat{I_s f} = (2\pi |\xi|)^{-s}\hat{f}, \quad
f \in \dot{H}^\al(\R).
\end{align*}
As the same method of the case of $\De^s$, we can induce the  result
$\eta_i * (I_s f) = I_s( \eta_i * f)$.
Hence, for $ f\in  \dot{H}^{\al} (\R), \,\, \al \in {\bf R}$,
\begin{align*}
\| I_s f\|^2_{ \dot{H}^{\al+s} (\R)} &= \sum_{-\infty < i <\infty}
       \int_{\R} |\xi|^{2(\al +s)}|\widehat{\eta_i *I_sf}(\xi)|^2 d\xi\\
     &= \sum_{-\infty < i <\infty}
       \int_{\R} |\xi|^{2(\al +s)}|\widehat{I_s( \eta_i *f)}(\xi)|^2 d\xi\\
     &= \sum_{-\infty < i <\infty}
       \int_{\R} |\xi|^{2\al}|\widehat{ \eta_i *f)}(\xi)|^2 d\xi \\
     & = \| f\|^2_{\dot H^\al (\R)}.
\end{align*}
Hence, $I_s: \dot{H}^{\al} (\R) \ri  \dot{H}^{\al +s} (\R)$ is isomorphism.
By the definition of fractional Laplacian of $I_{2s}f$, we have $\De^s I_{2s} f =f$
with distribution sense.

If $0 < s < n$ and $f \in \dot{H}^\al(\R), \,\, \al \geq 0$, then $I_sf$ is represented by
\begin{align}\label{boundary integral operator}
I_s f (x)  = \int_{ \R}   \Ga_s (x -y) f (y) dy \qquad x \in \R,
\end{align}
where
\begin{align}\label{riesz kernel}
\Ga_s (x) = c(n,s) \frac1{|x|^{n-s}}
\end{align}
is the Riesz potential of order $s$  in $\R$ (see section 4 in \cite{St}).
Since $I_{2s} f \in \dot{H}^s(\R)$ for $f \in \dot{H}^{-s}(\R)$,  by \eqref{dual}, we have
\begin{align}\label{0516}
\notag < f, I_{2s} f>  & = \sum_{-\infty< i<\infty}
  \int_{\R} \widehat{\eta_i * f}(\xi) \widehat{I_{2s} f}(\xi)  d\xi\\
\notag  & = \sum_{-\infty< i<\infty}  \sum_{-\infty< k < \infty}
  \int_{\R} \widehat{\eta_i * f}(\xi)  \widehat{   \eta_k * I_{2s} f}(\xi)  d\xi\\
  & = \sum_{-\infty< i<\infty}  \sum_{-\infty< k < \infty}
  \int_{\R} \widehat{\eta_i * f}(\xi) |2\pi \xi|^{-2s} \widehat{   \eta_k* f}(\xi)  d\xi\\
\notag &  = \sum_{-\infty< i<\infty}
  \int_{\R} |\xi|^{-2s}  |\widehat{\eta_i * f}(\xi)|^2  d\xi\\
\notag & \approx \| f\|^2_{\dot{H}^{-s} (\R)}.
  \end{align}

For $\phi \in    \dot{H}^{-s}_{ 0} (\Om^c)$, let us $u(x) :=  I_{2s} \tilde \phi(x)$.
Note that $\De^s u =\tilde \phi$. Hence, $u$ is weak solution of \eqref{main equation} and by (2) of remark \ref{rem-0221}, we have
$$\De^s u(x) =0, \qquad x \in \Om.
$$

We define bounded operator $  S_s : \dot{H}^{-s}_{ 0} (\Om^c) \ri  \dot{H}^s (\Om^c)$ by
\begin{align*}
S_s \phi :=  (I_{2s} \tilde \phi)|_{\Om^c}, \qquad \phi \in \dot{H}^{-s}_{ 0} (\Om^c).
\end{align*}

To prove theorem \ref{theo0229}, we prove the following theorem.
\begin{theo}\label{0112}
Let $\Om$ be a bounded Lipschitz domain in $\R$ and $0 <s < 1$. Then
\begin{align*}
 S_{s}: \dot{H}^{-s}_0 (\Om^c) \ri \dot{H}^{s} (\Om^c)
\end{align*}
is bijective.

\end{theo}

\section{Proof of Theorem \ref{0112} }
\setcounter{equation}{0}

To prove Theorem \ref{0112}, we need several following lemmas.

\begin{lemm}
\begin{align*}
S_{s} : \dot{H}^{-s}_{  0 }(\Om^c) \ri  \dot{H}^{s}(\Om^c)
\end{align*}
is one-to-one.
\end{lemm}
\begin{proof}
Suppose that $S_{s} f =0$ for $f \in  H^{-s}_{ 0}(\Om^c)$.
By definition $\tilde f$, we have $\tilde f \in H^{-s}(\R)$.
Then, by \eqref{0516}, we have
\begin{align}\label{1203}
0=< f,  S_{s} f>_{\Om^c} = < \tilde f, I_{2s} \tilde f>
\approx \| \tilde f\|^2_{\dot{H}^{-s} (\R)}.
\end{align}
Hence, we have $\tilde f = 0$ in $\dot{H}^{-s}(\R)$ and so  $\tilde
f$ is polynomial. Since $\tilde f =0$ in $\Om$, we get $\tilde f =0$
in $\R$. This implies that $f =0$. Hence, $S_{s} : \dot{H}^{-s}_{
0}(\Om^c) \ri  \dot{H}^{s}(\Om^c)$ is one-to-one.
\end{proof}

\begin{lemm}\label{closedrange}
\begin{align*}
S_{s} : \dot{H}^{-s}_{ 0 }(\Om^c) \ri  \dot{H}^{s}(\Om^c)
\end{align*}
has closed range.
\end{lemm}
\begin{proof}
From \eqref{1203} and \eqref{equiv}, we have
\begin{align*}
 \| f \|^2_{\dot{H}^{-s}_{ 0}( \Om^c)}
\leq c\|  \tilde f   \|^2_{\dot{H}^{-s}(\R)}
\leq c\| f\|_{\dot{H}^{-s}_{ 0}(\Om^c)}   \| S_{s} f\|_{\dot{H}^s(\Om^c)}.
\end{align*}
This implies that  $  S_{s} : \dot{H}^{-s}_{ 0 }(\Om^c) \ri  \dot{H}^{s}(\Om^c)$
has closed range.
\end{proof}

For the proof of bijectivity of $S_s$, it remains only to show
$S_s$ is onto.  Because of lemma \ref{closedrange}, we have
only to show that $S_s$ has dense range. Let $S^*_s $ be a dual operator of $S_s$,
 $Ker (S^*_s)$ be the kernel of $S^*_s$, $R(S_s)$ be the range of $S_s$
and
 $R(S_s)^{\bot}$ be an orthogonal complement of $R(S_s)$.
 Then there is a relation
$$
Ker (S^*_s ) = R(S_s)^{\bot} = \overline{R(S_s)^{\bot}}.
$$
Hence  $ S_s$ has dense range if and only if $S^*_s$ is
one-to-one.
Suppose that
$S^*_{s} \phi =0$ for some $\phi \in (  \dot{H}^{s}(\Om^c))^* =  \dot{H}^{ -s}_{ 0}(\Om^c) $.
Then, we have
\begin{align*}
0=<< S^*_{s}  \phi, \phi>> = <\phi, S_{s} \phi>
\approx \| \tilde \phi\|^2_{\dot{H}^{-s} (\R)}.
\end{align*}
This implies that $\tilde \phi =0$  and hence $\phi =0$. Therefore,
$S^*_{s} :(  \dot{H}^{s}(\Om^c))^* \ri  (  \dot{H}^{-s}_{ 0 }(\Om^c) )^*$ is one-to-one.

{\bf Uniqueness of solution.}
Suppose that $u \in \dot{H}^s(\R)$ is a weak solution of \eqref{main equation} such that $u =0$ in $\Om^c$.
Then, by \eqref{1013}, we have
\begin{align*}
\int_{\R} |\De^{\frac{s}2} u(x)|^2 dx =
\int_{\R} (2\pi  |\xi|)^{2s} |\hat{u}(\xi)|^2 d \xi =0.
\end{align*}
This means that $supp \, u = \{ 0\}$, that is, $u$ is polynomial. Since $u =0$ in $\Om^c$, we have
$u \equiv 0$ and so we proved the uniqueness of solution.
$\Box$

\section{Proof of Theorem \ref{theo0229-2} }
\setcounter{equation}{0}
To prove the theorem \ref{theo0229-2}, we need the following proposition and lemma.
\begin{prop}\label{prop}
Let $\Om$ be a bounded $C^{1,1}$ domain in $\R$ and $K$ be a
potential defined in \eqref{potentail0313}. Then, there are
positive constants $c_1, \,\, c_2>0$ such that for $x \in \Om$ and
$y \in \Om^c$
\begin{align*}
c_1 \frac{\de^s(x)}{\de^s(y)(\de(y) +1)^s|x-y|^n} \leq K(x,y)
\leq c_2\frac{\de^s(x)}{\de^s(y)(\de(y) +1)^s|x-y|^n},
\end{align*}
where $\de(z) = dist(z,\pa \Om)$ for $z \in \R \setminus \pa \Om$.
\end{prop}
(see theorem 1.5 in \cite{CS}).

\begin{lemm}\label{lemma0313}
Let $0 < s <   \frac{n}{2(n-1)}$ and $u$ be $2s$-harmonic in bounded $C^{1,1}$ domain $\Om$.
If $u \in L^{\frac{2n}{n-2s}}(\Om)$, then
$u$ is regular $2s$-harmonic.
\end{lemm}
\begin{proof}
Let $\Om_k, \,\, 1 \leq k < \infty$ be subset of $\Om$ satisfying
that $\Om_k$ are $C^{1,1}$ domains for all $k$ such that $\bar \Om_k
\subset \Om_{k+1}$ and $\cup_{1\leq k <\infty} \Om_k = \Om$. Let $x
\in \Om$. Then for some $k_0$ such that $x \in \Om_{k_0}$. Note that
$u$ is regular  $\al$-harmonic in $\Om_k$. Hence,  for $k \geq k_0$,
we have
\begin{align*}
u(x) = E_x(u(X_{\tau_{\Om_k}})) & = E_x[u(X_{\tau_{\Om_k}}); X_{\tau_{\Om_k}} \in \Om \setminus \Om_k]
+ E_x[u(X_{\tau_{\Om_k}}); X_{\tau_{\Om_k}} \in  \Om^c]\\
& = E_x[u(X_{\tau_{\Om_k}}); X_{\tau_{\Om_k}} \in \Om \setminus \Om_k]
+ E_x[u(X_{\tau_{\Om}}); X_{\tau_{\Om_k}} \in  \Om^c]
\end{align*}
Here, the semicolon above means as usual that the integration is over the subsequent set.
Since $\tau_{\Om_k} \ri \tau_{\Om}$ a.s and
$\lim_{k \ri \infty} P^x(X_{\tau_{\Om_k}} = X_{\tau_{\Om}}) =1$
(see (5.40) in \cite{B}), the second term goes to
\begin{align*}
E_x[u(X_{\tau_{\Om}}); X_{\tau_{\Om}} \in  \Om^c].
\end{align*}
For the first term, we denote $\de_k(x) = dist(x,\pa \Om_k)$,
$r_k = dist(\pa \Om_k, \pa \Om)$
and $K_k$ is potential defined in
\eqref{potentail0313} replaced by $\Om_k \times (\Om_k)_-$.
Using the  representation of $2s$-harmonic function, proposition \ref{prop} and Holder inequality,
the first term is
\begin{align}\label{equation0313}
\begin{array}{ll}\vspace{2mm}
|\int_{\Om \setminus \Om_k} K_k(x,y) u(y) dy|
& \leq c \int_{\Om \setminus \Om_k}
\frac{\de_k^s(x)}{\de_k^s (y) (\de_k(y) +1)^s |x-y|^n} |u(y)| dy\\ \vspace{2mm}
& \leq c \de_k^s(x) \| u\|_{L^{\frac{2n}{n -2s}}(\Om \setminus \Om_k)}
\Big(\int_{\Om \setminus \Om_k}
\frac{1}{\de_k^{\frac{2n}{n+ 2s} s} (y)  |x-y|^{\frac{2n}{n+2s}n}}
dy\Big)^{\frac{n+2s}{2ns}}\\ \vspace{2mm}
& \leq c \de_k^{-n +s}(x)\| u\|_{L^{\frac{2n}{n -2s}}(\Om)} \Big(\int_0^{r_k}
 r^{-{\frac{2ns}{n+ 2s}   }}  dr\Big)^{\frac{n+2s}{2ns}}\\
& \leq c \de_k^{-n +s}(x)\| u\|_{L^{\frac{2n}{n -2s}}(\Om)}
 r_k^{-1 + \frac{n+ 2s}{2ns}   }.
 \end{array}
\end{align}
Since $\de_k(x) \ri \de(x)$,  $r_k \ri 0$ and $ -1 + \frac{n+ 2s}{2ns}   > 0$, the last term goes to
zero. Hence, for $x \in \Om$, we get
\begin{align*}
u(x) = E_x(u(X_{\tau_\Om}))
\end{align*}
and so $u$ is regular $2s$-harmonic in $\Om$.
\end{proof}

{\bf Proof of Theorem \ref{theo0229-2}.}
In  \eqref{equation0313}, using the Sobolev inequality, we have
\begin{align*}
|\int_{\Om \setminus \Om_k} K_k(x,y) u(y) dy|
 \leq c \de_k^{-n +s}(x)\| u\|_{H^s(\Om)}
 r_k^{-1 + \frac{n+ 2s}{2ns}   } \ri 0 \quad \mbox{as} \quad k \ri \infty.
\end{align*}
Hence, by the process of the proof of Lemma \ref{lemma0313}, we get
$u$ is regular $2s$-harmonic. $\Box$

\end{document}